\documentclass{amsart}
\usepackage{amsmath}
\usepackage{amsfonts}
\usepackage{amssymb}
\usepackage[utf8]{inputenc}
\usepackage{amssymb}
\usepackage{mathrsfs}
\usepackage{amsthm}

\usepackage{enumerate}

\usepackage{graphicx}

\newtheorem{thrm}{Theorem}[section]
\newtheorem{cor}[thrm]{Corollary}
\newtheorem{lemma}[thrm]{Lemma}
\newtheorem{prop}[thrm]{Proposition}

\theoremstyle{definition}

\newtheorem{rem}[thrm]{Remark}

\newtheorem{claim}{Claim}

\newtheorem*{xrem}{Remark}

\newcommand{\NN}{\mathbb{N}}

\newcommand{\la}{\langle}
\newcommand{\ra}{\rangle}

\begin{document}

\title[On $\kappa$-pseudocompactess]{On $\kappa$-pseudocompactess and uniform homeomorphisms of function spaces}
\author[M. Krupski]{Miko\l aj Krupski}
\thanks{The author was partially supported by the NCN (National Science Centre, Poland) research Grant no. 2020/37/B/ST1/02613}
\date{}
\address{
Institute of Mathematics\\ University of Warsaw\\ ul. Banacha 2\\
02--097 Warszawa, Poland }
\email{mkrupski@mimuw.edu.pl}

\begin{abstract}
A Tychonoff space $X$ is called \textit{$\kappa$-pseudocompact} if for every continuous mapping $f$ of $X$ into $\mathbb{R}^\kappa$
the image $f(X)$ is compact. This notion generalizes pseudocompactness and gives a stratification of spaces lying between pseudocompact and compact spaces.
It is well known that pseudocompactness of $X$ is determined by the uniform structure
of the function space $C_p(X)$ of continuous real-valued functions on $X$ endowed with the
pointwise topology. In respect of that
A.V. Arhangel'skii asked in [Topology Appl., 89 (1998)] if analogous assertion is true for
$\kappa$-pseudocompactness. We provide an affirmative answer to this question.
\end{abstract}



\maketitle

\section{Introduction}
In this note, by a space we mean a Tychonoff topological space.
For a space $X$, by $C_p(X)$ we denote the space of all continuous real-valued functions on $X$ endowed with the pointwise topology.
The symbol $C_p^\ast(X)$ stands for the subspace of $C_p(X)$ consisting of all \textit{bounded} continuous functions.
Recall that $X$ is \textit{pseudocompact} if $C_p(X)=C_p^\ast(X)$, i.e. every real-valued continuous function on $X$ is bounded.
In 1962 J.F. Kennison \cite{Ke} introduced the following generalization of psudocompactness.
Let $\kappa$ be an infinite cardinal. A space $X$ is called \textit{$\kappa$-pseudocompact} if for every continuous mapping $f$ of $X$ into $\mathbb{R}^\kappa$
the image $f(X)$ is compact.
Clearly, $\kappa$-pseudocompactness implies $\lambda$-pseudo\-compact\-ness for every infinite cardinal $\lambda\leq\kappa$.
It is also clear that $\omega$-pseudocompactness is precisely pseudocompactness. In particular any $\kappa$-pseudocompact space is pseudocompact.
It was established by 
Uspenski\u{\i} in \cite{U} that pseudocompactness of $X$ is determined by the uniform structure of the function space $C_p(X)$ (see \cite{GKM}
for a different proof of this result).
In respect of that A.V. Arhangel'skii asked in 1998 if analogous result holds for $\kappa$-pseudocompactness
(see \cite[Question 13]{Ar} or \cite[Problem 4.4.2]{Tk}).

The aim of the present note is to
provide an affirmative answer to this question by proving the
following extension of Uspenski\u{\i}'s theorem:

\begin{thrm}\label{theorem_main}
For any infinite cardinal $\kappa$, if $C_p(X)$ and $C_p(Y)$ are uniformly homeomorphic,
then $X$ is $\kappa$-pseudocompact if and only if $Y$ is $\kappa$-pseudocompact.
\end{thrm}
 
Let us recall that that a map $\varphi:C_p(X)\to C_p(Y)$
is \textit{uniformly continuous} if for each open neighborhood $U$ of the zero function in $C_p(Y)$, there is and open neighborhood $V$ of the zero function
in $C_p(X)$ such that $(f-g)\in V$ implies $(\varphi(f)-\varphi(g))\in U$.
Spaces $C_p(X)$ and $C_p(Y)$ are \textit{uniformly homeomorphic} if there is a homeomorphism $\varphi$ between them such that both $\varphi$ and $\varphi^{-1}$
are uniformly continuous.

The proof of Theorem \ref{theorem_main} is inspired by author's recent work \cite{Kr} concerned with linear homeomorphisms of function spaces.
The basic idea in \cite{Kr} relies on the fact that
certain topological properties of a space $X$ can be conveniently characterized by the way $X$ is positioned in its \v{C}ech-Stone compactification $\beta X$;
$\kappa$-pseudocompactess is one of such properties. Indeed, Hewitt \cite{H}  gave the following description of pseudocompactness
(cf. \cite[Theorem 1.3.3]{pseudocompact}).

\begin{thrm}(Hewitt)\label{theorem_Hewitt}
A space $X$ is pseudocompact if and only if every nonempty $G_\delta$-subset of $\beta X$ meets $X$.
\end{thrm}
It was noted by Retta in \cite{Re} that the above result easily extends to $\kappa$-pseudo\-compact\-ness. We need the following notation.
Let $\kappa$ be an infinite cardinal. A subset $A$ of a space $Z$ is a \textit{$G_\kappa$-set} if it is an intersection
of at most $\kappa$-many open subsets of $Z$. The $G_{\omega}$-sets are called $G_\delta$-sets and the complement of a $G_\delta$-set is called $F_\sigma$-set.
We have (see \cite[Theorem 1]{Re}):

\begin{thrm}(Retta)\label{Retta}
Let $\kappa$ be an infinite cardinal. A space $X$ is $\kappa$-pseudocompact if and only if every nonempty $G_\kappa$-subset of $\beta X$ meets $X$.
\end{thrm}

The uniform structure of spaces of continuous functions was studied by many authors; the interested reader should consult the book \cite{Tk}.
For our purposes, the most important are some ideas developed by Gul'ko in \cite{G}.


\section{Results}


For a space $Z$ and a function $f\in C^\ast_p(Z)$
the function $\widetilde{f}:\beta Z\to \mathbb{R}$ is the unique continuous extension of $f$ over the \v{C}ech-Stone compactification $\beta Z$ of $Z$.
Let $\varphi:C^\ast_p(X)\to C^\ast_p(Y)$ be a uniformly continuous surjection.
For $y\in \beta Y$ and a subset $K$ of $\beta X$ we define

\begin{align*}
a(y,K)=\sup\{|\widetilde{\varphi(f)}(y)-\widetilde{\varphi(g)}(y)|&:\;f,g\in C_p^\ast(X) \mbox{ such that }\\
&|\widetilde{f}(x)-\widetilde{g}(x)|<1 \mbox{ for every }x\in K\}
\end{align*}
Note that $a(y,\emptyset)=\infty$ since $\varphi$ is onto.

For $y\in \beta Y$ define the family
$$\mathscr{A}(y)=\{K\subseteq \beta X: K \mbox{ is compact and } a(y,K)<\infty\}.$$
Similarly, for $y\in \beta Y$ and $n\in \mathbb{N}$ let
$$\mathscr{A}_n(y)=\{K\subseteq \beta X: K \mbox{ is compact and } a(y,K)\leq n\}.$$

It may happen that for some $n$ the family $\mathscr{A}_n(y)$ is empty. However, we have the following:

\begin{prop}\label{A_n is nonempty}
For every $y\in Y$, there exists $n$ for which $\mathscr{A}_n(y)$ contains a nonempty finite subset of $X$.
In particular, for this $n$ the family $\mathscr{A}_n(y)$ is nonempty.
\end{prop}

\begin{proof}
By uniform continuity of $\varphi$, there is $n\in \mathbb{N}$ and a finite subset $F$ of $X$ such that
\begin{equation}\label{equation 1}
    \begin{aligned}
    &\mbox{if } |f(x)-g(x)|<1/n \mbox{ for every } x\in F,\\ &\mbox{then } |\varphi(f)(y)-\varphi(g)(y)|<1.
    \end{aligned}
\end{equation}
We claim that $F\in \mathscr{A}_n(y)$. To see this, take arbitrary functions $f,g\in C_p^\ast(X)$ such that $|f(x)-g(x)|<1$, for every $x\in F$.
Put $f_k=f+\tfrac{k}{n}(g-f)$, for $k=0,1,\ldots , n$. Then $f_0=f$, $f_n=g$ and $|f_k(x)-f_{k+1}(x)|<1/n$ for $x\in F$.
Hence, by \eqref{equation 1} we get
\begin{align*}
&|\varphi(f)(y)-\varphi(g)(y)|\leq\\
&|\varphi(f_0)(y)-\varphi(f_1)(y)|+\ldots + |\varphi(f_{n-1})(y)-\varphi(f_n)(y)|<n,
\end{align*}
as required.
\end{proof}

Clearly, for every $y\in \beta Y$ we have $\mathscr{A}(y)=\bigcup_{n\in \mathbb{N}}\mathscr{A}_n(y)$. In particular,
for $y\in Y$ the family $\mathscr{A}(y)$ is always nonempty.

For $n\in \NN$ we set
$$Y_n=\{y\in \beta Y:\mathscr{A}_n(y) \mbox{ is nonempty}\}.$$
Note that $y\in Y_n$ if and only if $\beta X\in \mathscr{A}_n(y)$. Using this observation it is easy to show the following:
\begin{lemma}\label{Y_n is compact}
For every $n\in \NN$ the set $Y_n$ is closed in $\beta Y$; hence compact.
\end{lemma}
\begin{proof}
Pick $y\in \beta Y\setminus Y_n$. Since $\beta X\notin \mathscr{A}_n(y)$, there are functions $f,g\in C_p^\ast(X)$ satisfying $|\widetilde{f}(x)-\widetilde{g}(x)|<1$ for every $x\in \beta X$, and $|\widetilde{\varphi(f)}(y)-\widetilde{\varphi(g)}(y)|>n$. The set
$$U=\{z\in \beta Y:|\widetilde{\varphi(f)}(z)-\widetilde{\varphi(g)}(z)|>n\}$$
is an open neighborhood of $y$ in $\beta Y$. Moreover if $z\in U$, then $f$ and $g$ witness that $\beta X\notin \mathscr{A}_n(z)$; thus $U\cap Y_n=\emptyset$.
\end{proof}

For a space $X$ and a positive integer $m$, we denote by $[X]^{\leq m}$ the space of all nonempty at most $m$-element subsets of $X$ endowed with the Vietoris topology, i.e. basic open sets in $[X]^{\leq m}$ are of the form
$$\la U_1,\ldots ,U_k\ra=\{F\in [X]^{\leq m}:\forall i\leq k\quad F\cap U_i\neq \emptyset\; \text{and}\;F\subseteq \bigcup_{i=1}^k U_i\},$$
where $\{U_1,\ldots U_k\}$ is a finite collection of open subset of $X$.

For any positive integers $n,m$ we define
$$Y_{n,m}=\{y\in \beta Y:\mathscr{A}_n(y)
\cap[\beta X]^{\leq m}\neq\emptyset\}$$
Note that $Y_{n,m}\subseteq Y_n$ and by Proposition \ref{A_n is nonempty} we have
\begin{equation}
Y\subseteq \bigcup_{n,m}Y_{n,m}   \label{equation(2)}
\end{equation}
We claim that $Y_{n,m}$ is closed in $Y_n$ and hence it is compact:
\begin{lemma}\label{lemma_Y_nm}
The set $Y_{n,m}$ is closed in $Y_n$, hence it is compact.
\end{lemma}
\begin{proof}
Consider the following subset $Z$ of the product $Y_n\times [\beta X]^{\leq m}$
$$Z=\{(y,F)\in Y_n\times [\beta X]^{\leq m}:F\in \mathscr{A}_n(y)\}.$$
We show that $Z$ is closed. Pick $(y,F)\in (Y_n\times [\beta X]^{\leq m})\setminus Z$. Then $F\notin \mathscr{A}_n(y)$ and thus there are $f,g\in C_p^\ast(X)$ satisfying $|\widetilde{f}(x)-\widetilde{g}(x)|<1$ for every $x\in F$, and $|\widetilde{\varphi(f)}(y)-\widetilde{\varphi(g)}(y)|>n$.
Let $U=\{x\in \beta X:|\widetilde{f}(x)-\widetilde{g}(x)|<1\}$ and
$V=\{z\in Y_n:|\widetilde{\varphi(f)}(z)-\widetilde{\varphi(g)}(z)|>n\}$. The set $V\times \la U \ra$ is an open neighborhood of $(y,F)$ in $Y_n\times [\beta X]^{\leq m}$ disjoint from $Z$.

The set $Z$, being is closed in the compact space $Y_n\times [\beta X]^{\leq m}$, is compact.
Since the set $Y_{n,m}$ is the image of $Z$ under the projection map it must be compact.
\end{proof}

\begin{cor}\label{betaY=sum of Y_nm}
Suppose that $Y$ is pseudocompact and let $\varphi:C_p^\ast(X)\to C_p^\ast(Y)$ be a uniformly continuous surjection. For every $y\in \beta Y$, there exist $n$ and $m$ such that $y\in Y_{n,m}$.
\end{cor}
\begin{proof}
By Lemma \ref{lemma_Y_nm}, the set $\bigcup_{n,m} Y_{n,m}$ is $F_\sigma$ in $\beta Y$ and contains $Y$, by \eqref{equation(2)}. It follows from Theorem \ref{theorem_Hewitt} that $\bigcup_{n,m}Y_{n,m}=\beta Y$.
\end{proof}

For $y\in \bigcup_{n,m} Y_{n,m}$ we define

$$K(y)=\bigcap \mathscr{A}(y).$$

\begin{xrem}
For $y\in Y$ the set $K(y)$ is the support introduced by Gul'ko in \cite{G} (see also \cite{MP}, \cite{A}, \cite{GKM}).
\end{xrem}

\begin{lemma}\label{Lemma K(y)}
For every $y\in \bigcup_{n,m} Y_{n,m}$ the set $K(y)$ is a nonempty finite subset of $\beta X$. Moreover, $K(y)\in \mathscr{A}(y)$. If $y\in Y$, then $K(y)$ is a
subset of $X$.
\end{lemma}
\begin{proof}
We show that the family $\mathscr{A}(y)$ is closed under finite intersections. Pick $K_1,K_2\in \mathscr{A}(y)$
and let $f,g\in C_p^\ast(X)$ be such that $|\widetilde{f}(x)-\widetilde{g}(x)|<1$ for every $x\in K_1\cap K_2$.
Let
$$U=\{x\in \beta X: |\widetilde{f}(x)-\widetilde{g}(x)|<1\}.$$
The set $U$ is open in $\beta X$ and $K_1\cap K_2\subseteq U$.

Since $K_1$ and $K_2\setminus U$ are disjoint closed subsets of the compact space $\beta X$, by Urysohn's lemma
there is a continuous function $u:\beta X\to [0,1]$ such that
\begin{equation*}
u(x)=
  \left\{\begin{aligned}
  &   1 &&\mbox{for }x\in K_1\\
&0 \quad &&\mbox{for }x\in K_2\setminus U
  \end{aligned}
 \right. 
\end{equation*}

Let

$$\widetilde{h}=u\cdot (\widetilde{f}-\widetilde{g})+\widetilde{g}$$

and let $h\in C^\ast_p(X)$ be the restriction of $\widetilde{h}$ to $X$. We have:
\begin{itemize}
 \item $\widetilde{h}(x)=\widetilde{f}(x)$ for $x\in K_1$,
 \item $\widetilde{h}(x)=\widetilde{g}(x)$ for $x\in K_2\setminus U$ and
 \item if $x\in U$, then $|\widetilde{h}(x)-\widetilde{g}(x)|=|u(x)|\cdot |\widetilde{f}(x)-\widetilde{g}(x)|<1$,
 by definition of $U$ and the fact that $u$ maps into $[0,1]$.
\end{itemize}
In particular, since $K_1\cap K_2\subseteq U$, we get
\begin{itemize}
 \item $|\widetilde{h}(x)-\widetilde{g}(x)|<1$ for $x\in K_2$.
\end{itemize}

Since $K_1\in \mathscr{A}(y)$ and $\widetilde{h}(x)=\widetilde{f}(x)$ for $x\in K_1$,
we get $|\widetilde{\varphi(f)}(y)-\widetilde{\varphi(h)}(y)|\leq a(y,K_1)<\infty$.
Similarly, since $|\widetilde{h}(x)-\widetilde{g}(x)|<1$ for $x\in K_2$ and $K_2\in \mathscr{A}(y)$,
we have $|\widetilde{\varphi(g)}(y)-\widetilde{\varphi(h)}(y)|\leq a(y,K_2)<\infty$. Hence,
$$|\widetilde{\varphi(f)}(y)-\widetilde{\varphi(g)}(y)|\leq |\widetilde{\varphi(f)}(y)-\widetilde{\varphi(h)}(y)|
+ |\widetilde{\varphi(h)}(y)-\widetilde{\varphi(g)}(y)|\leq a(y,K_1)+a(y,K_2).$$
So
$a(y,K_1\cap K_2)\leq a(y,K_1)+a(y,K_2)$ and thus
$K_1\cap K_2\in \mathscr{A}(y)$. By induction the result follows for any finite intersection.

The family $\mathscr{A}(y)$, consisting of closed subsets of $\beta X$, is closed under finite intersections and $\emptyset \notin \mathscr{A}(y)$
so by compactness the intersection $\bigcap \mathscr{A}(y)$ must be nonempty.
It is finite because $y\in Y_{n,m}$ guarantees that the family $\mathscr{A}(y)$ contains a subset of $\beta X$ which is at most $m$-element.

Since $\mathscr{A}(y)$ contains a finite subset $F$ of $\beta X$, the set $K(y)\subseteq F$ is an intersection of finitely many elements of
$\mathscr{A}(y)$ so the first part of the proof implies that $K(y)\in \mathscr{A}(y)$.

Finally, if $y\in Y$ then $y\in \bigcup_{n,m}Y_{n,m}$, by \eqref{equation(2)}. So $K(y)$ is well defined. The inclusion
$K(y)\subseteq X$ follows from Proposition \ref{A_n is nonempty}.
\end{proof}

For $y\in \bigcup_{n,m} Y_{n,m}$ we define
$$a(y)=a(y,K(y))$$
By the Lemma \ref{Lemma K(y)}, $K(y)\in \mathscr{A}(y)$ so $a(y)<\infty$.
For a subset $A$ of $\beta X$ we set
$$K^{-1}(A)=\left\{y\in \bigcup_{n,m} Y_{n,m}:K(y)\cap A\neq \emptyset\right\}.$$

Combining Corollary \ref{betaY=sum of Y_nm} and Lemma \ref{Lemma K(y)} we get

\begin{prop}\label{K(y) is finite}
Suppose that $\varphi:C^\ast_p(X)\to C^\ast_p(Y)$ is a uniformly continuous surjection. If $Y$ is pseudocompact then, for every $y\in \beta Y$,
the set $K(y)$ is a well-defined nonempty finite subset of $\beta X$ that belongs to the family $\mathscr{A}(y)$. Also, $a(y)$ is a well-defined number,
for every $y\in \beta Y$.
\end{prop}

The proof of the next lemma is analogous to the proof of \cite[Lemma 1.3]{A}.

\begin{lemma}\label{Lemma-Arbit}
 Suppose that $U\subseteq \beta X$ is open and let $n$ be a positive integer. For every $y\in K^{-1}(U)\cap Y_n$ there exists an open neighborhood $V$ of $y$ in
 $\beta Y$ such that for every $z\in V\cap Y_n$ and every $A\in \mathscr{A}_n(z)$ we have $A\cap U\neq \emptyset$.
\end{lemma}
\begin{proof}
 Fix $x_0\in K(y)\cap U$ witnessing $y\in K^{-1}(U)$. Since $K(y)$ is finite, shrinking $U$ if necessary we can assume that $U\cap K(y)=\{x_0\}$. Note that
 $\beta X\setminus U\notin \mathscr{A}(y)$ for otherwise $K(y)$ would be a subset of $\beta X\setminus U$ and this is not the case because $x_0\in K(y)\cap U$.
 It follows that there are $f,g\in C^\ast_p(X)$ such that
 \begin{align}
  &|\widetilde{f}(x)-\widetilde{g}(x)|<1 \mbox{ for every } x\in \beta X\setminus U\mbox{ and}\label{(3)}\\
  &|\widetilde{\varphi(f)}(y)-\widetilde{\varphi(g)}(y)|>a(y)+n \label{(4)}
 \end{align}
 Let $\widetilde{h}\in C_p(\beta X)$ be a function satisfying
 \begin{equation}
  \widetilde{h}(x)=\widetilde{f}(x) \mbox{ for every } x\in \beta X\setminus U,\mbox{ and } \widetilde{h}(x_0)=\widetilde{g}(x_0), \label{(5)}
 \end{equation}
 and let $h\in C^\ast_p(X)$ be the restriction of $\widetilde{h}$.
 
 Note that by \eqref{(3)} and \eqref{(5)},
 $|\widetilde{h}(x)-\widetilde{g}(x)|<1$ for every $x\in K(y)$. Therefore,
 \begin{equation}
  |\widetilde{\varphi(h)}(y)-\widetilde{\varphi(g)}(y)|\leq a(y). \label{(6)}
 \end{equation}
 According to \eqref{(4)} and \eqref{(6)} we have
 \begin{equation}
  |\widetilde{\varphi(f)}(y)-\widetilde{\varphi(h)}(y)|  \geq  |\widetilde{\varphi(f)}(y)-\widetilde{\varphi(g)}(y)|-
  |\widetilde{\varphi(h)}(y)-\widetilde{\varphi(g)}(y)|>n. \label{(7)}
 \end{equation}
 
 Let
 $$V=\{z\in \beta Y:|\widetilde{\varphi(f)}(z)-\widetilde{\varphi(h)}(z)|>n\}.$$
 The set $V$ is open and $y\in V$, by \eqref{(7)}. We show that $V$ is as required.
 
 Take $z\in V\cap Y_n$ and let $A\in \mathscr{A}_n(z)$. If $A\cap U=\emptyset$ then
 $\widetilde{h}(x)=\widetilde{f}(x)$ for every $x\in A$, by \eqref{(5)}. So $|\widetilde{\varphi(f)}(z)-\widetilde{\varphi(h)}(z)|\leq n$,
 contradicting $z\in V$.
\end{proof}
\begin{prop}\label{Corollary-Arbit}
 Suppose that $Y$ is pseudocompact. If $U\subseteq \beta X$ is open, then the set $K^{-1}(U)$ is a $G_\delta$-subset of $\beta Y$. 
\end{prop}
\begin{proof}
 By Proposition \ref{K(y) is finite}, for every $y\in \beta Y$, the set $K(y)$ is a nonempty finite subset of $\beta X$.
 For $n=1,2,\ldots$, let
 $$L_n=K^{-1}(U)\cap Y_n.$$
 For $y\in L_n$ let $V_n^y$ be an open neighborhood of $y$ in $\beta Y$ provided by Lemma \ref{Lemma-Arbit}, i.e.
 \begin{equation}
 \mbox{if } z\in V_n^y\cap Y_n \mbox{ and } A\in \mathscr{A}_n(z), \mbox{ then } A\cap U\neq \emptyset. \label{(8)}
 \end{equation}
 Let 
 $$V_n=\bigcup\{V_n^y:y\in L_n\}.$$
 
 We claim that
 $$K^{-1}(U)=\bigcap_{m=1}^\infty\bigcup_{n=m}^\infty V_n.$$
 Indeed, pick $y\in K^{-1}(U)$ and fix an arbitrary $m\geq 1$. Since $\beta Y=\bigcup_{n=1}^\infty Y_n$ (cf. Corollary \ref{betaY=sum of Y_nm}),
 there is $i$ such that $y\in Y_i$. Since $Y_n\subseteq Y_{n+1}$, we can assume that $i>m$.
 We have $y\in L_i$ whence $y\in V^y_i\subseteq V_i\subseteq \bigcup_{n=m}^\infty V_n$, because $i>m$.
 
 To prove the opposite inclusion, take
 $z\in\bigcap_{m=1}^\infty\bigcup_{n=m}^\infty V_n$. Again, there is $i$ such that $z\in Y_i$. Let $j$ be a positive integer satisfying $j>\max\{a(z),i\}$.
 By our assumption, $z\in\bigcup_{n=j}^\infty V_n$, so there is $k\geq j$ such that $z\in V_k$. Clearly, $z\in Y_k$ and since $k>a(z)$ we have
 \begin{equation}
  K(z)\in \mathscr{A}_k(z). \label{(9)}
 \end{equation}
 By definition of $V_k$, there is $y\in L_k$ such that $z\in V^y_k$. Now, from \eqref{(8)} and \eqref{(9)} we get
 $z\in K^{-1}(U)$.
\end{proof}

\begin{rem}
If $\varphi:C_p(X)\to C_p(Y)$ is a uniform homeomorphism, we may consider the inverse map $\varphi^{-1}:C_p(Y)\to C_p(X)$ and apply all
of the above results to $\varphi^{-1}$. In particular, if $X$ is pseudocompact, then for every $x\in \beta X$ we can define the set $K(x)\subseteq \beta Y$
and the real number $a(x)$
simply by
interchanging the roles of $X$ and $Y$ above.
\end{rem}

\begin{lemma}\label{lemma-back}
Suppose that both $X$ and $Y$ are pseudocompact spaces. Let $\varphi:C^\ast_p(X)\to C^\ast_p(Y)$ be a uniform homeomorphism. For any $x\in \beta X$ there
is $y\in K(x)$ such that $x\in K(y)$.
\end{lemma}
\begin{proof}
 Let $x\in \beta X$. Applying Proposition \ref{K(y) is finite}, first to $\varphi^{-1}$ and then to $\varphi$, we infer that the set
 $F=\bigcup\{K(y):y\in K(x)\}\subseteq \beta X$
 is finite being a finite union of finite sets. Let $M$ be a positive integer such that
$$M>\max\{a(y):y\in K(x)\}.$$
 
 Striving for a contradiction, suppose that $x\notin F$.
 Let $\widetilde{f},\widetilde{g}\in C_p(\beta X)$ be functions satisfying
 \begin{equation}
  \widetilde{f}(z)=\widetilde{g}(z) \mbox{ for every }z\in F \mbox{ and } |\widetilde{f}(x)-\widetilde{g}(x)|> M\cdot a(x). \label{(*)}
 \end{equation}
 Let $f\in C^\ast_p(X)$ and $g\in C^\ast_p(X)$ be the restrictions of $\widetilde{f}$ and $\widetilde{g}$, respectively.
Since for every $y\in K(x)$ the functions $\widetilde{f}$ and $\widetilde{g}$ agree on $K(y)\subseteq F$, we have
\begin{equation}
 |\widetilde{\varphi(f)}(y)-\widetilde{\varphi(g)}(y)|\leq a(y)<M,\mbox{ for every }y\in K(x). \label{(**)}
\end{equation}

For $k\in\{0,1,\ldots, M\}$ define a function $\widetilde{h_k}\in C_p(\beta Y)$ by the formula
$$\widetilde{h_k}=\widetilde{\varphi(f)}+\tfrac{k}{M}\left(\widetilde{\varphi(g)}-\widetilde{\varphi(f)}\right).$$
Obviously, $\widetilde{h_0}=\widetilde{\varphi(f)}$ and $\widetilde{h_M}=\widetilde{\varphi(g)}$.
Moreover, by \eqref{(**)}, we have
\begin{equation}
|\widetilde{h_{k+1}}(y)-\widetilde{h_{k}}(y)|=\tfrac{1}{M}|\widetilde{\varphi(g)}(y)-\widetilde{\varphi(f)}(y)|<1, \mbox{ for every } y\in K(x).
\label{12}
\end{equation}
For $k\in \{0,1,\ldots , M\}$ let $h_k\in C^\ast_p(Y)$ be the restriction of $\widetilde{h_k}$. Using \eqref{12} we get:
\begin{align*}
 |\widetilde{f}(x)-\widetilde{g}(x)|=|\widetilde{\varphi^{-1}(\varphi(f))}(x)-\widetilde{\varphi^{-1}(\varphi(g))}(x)|=
 |\widetilde{\varphi^{-1}(h_0)}(x)-\widetilde{\varphi^{-1}(h_M)}(x)|\leq \\
 |\widetilde{\varphi^{-1}(h_0)}(x)-\widetilde{\varphi^{-1}(h_1)}(x)|+\ldots + |\widetilde{\varphi^{-1}(h_{M-1})}(x)-\widetilde{\varphi^{-1}(h_M)}(x)|\leq M\cdot a(x)
\end{align*}
This however contradicts \eqref{(*)}.
\end{proof}

Now we are ready to prove of our main result.

\begin{proof}[Proof of Theorem \ref{theorem_main}]
Let $\kappa$ be an infinite cardinal and let $\varphi:C_p(X)\to C_p(Y)$ be a uniform homeomorphism. By symmetry it is enough to show that
if $Y$ is $\kappa$-pseudocompact then so is $X$. So let us assume that $Y$ is $\kappa$-pseudocompact. Then, in particular $Y$ is pseudocompact and hence by
Uspenski\u{\i}'s theorem \cite[Corollary]{U} (cf. \cite[V.136]{Tk}), so is $X$. Hence, $C_p(Y)=C^\ast_p(Y)$ and $C_p(X)=C^\ast_p(X)$. In order to prove that $X$ is $\kappa$-pseudocompact
we will employ Theorem \ref{Retta}. For this purpose fix a nonempty $G_\kappa$-subset $G$ of $\beta X$. It suffices to prove that $G\cap X\neq \emptyset$.

\begin{claim}\label{Claim 1}
The set $K^{-1}(G)=\{y\in \beta Y:K(y)\cap G\neq \emptyset\}$ is nonempty.
\end{claim}
\begin{proof}
 The set $G$ is nonempty so let us fix $x\in G$. According to Lemma \ref{lemma-back} there is $y\in K(x)$ such that $x\in K(y)$. In particular, $y\in K^{-1}(G)$.
\end{proof}

\begin{claim}\label{Claim 2}
The set $K^{-1}(G)=\{y\in \beta Y:K(y)\cap G \neq \emptyset\}$ is a $G_\kappa$-set in $\beta Y$. 
\end{claim}
\begin{proof}
Write $G=\bigcap\{U_\alpha:\alpha<\kappa\}$, where
each $U_\alpha$ is an open subset of $\beta X$. We can also assume that the family $\{U_\alpha:\alpha<\kappa\}$ is closed under finite intersections.
It follows from Proposition \ref{Corollary-Arbit} that for each $\alpha<\kappa$, the set $K^{-1}(U_\alpha)$ is $G_\delta$ in $\beta Y$.
Thus, it is enough to show that
$$K^{-1}(G)=\bigcap_{\alpha < \kappa}K^{-1}(U_\alpha).$$
To this end, take $y\in \bigcap_{\alpha < \kappa}K^{-1}(U_\alpha)$. According to Proposition \ref{K(y) is finite}, the set $K(y)$ is a nonempty
finite subset of $\beta X$. Enumerate $K(y)=\{x_1,\ldots,x_k\}$, where $k$ is a positive integer. If $y\notin K^{-1}(G)$, then
for every $i\leq k$ there is $\alpha_i<\kappa$ such that
\begin{equation}
x_i\notin U_{\alpha_i}. \label{13}
\end{equation}
The family $\{U_\alpha:\alpha<\kappa\}$ is closed under finite intersections, so
there is $\gamma<\kappa$ with $U_\gamma=U_{\alpha_1}\cap\ldots\cap U_{\alpha_k}$. But $y\in \bigcap_{\alpha < \kappa}K^{-1}(U_\alpha)\subseteq K^{-1}(U_\gamma)$.
Hence, there is $j\leq k$ such that $x_j\in U_\gamma\subseteq U_{\alpha_j}$, which is a contradiction with \eqref{13}. Therefore, we must have $y\in K^{-1}(G)$. This provides
the inclusion $K^{-1}(G)\supseteq\bigcap_{\alpha < \kappa}K^{-1}(U_\alpha)$. The opposite inclusion is immediate.
\end{proof}

It follows from Claims \ref{Claim 1} and \ref{Claim 2} that the $K^{-1}(G)$ is a nonempty $G_\kappa$-subset of $\beta Y$. Hence, by Theorem \ref{Retta},
there exists
$p\in K^{-1}(G)\cap Y$. We have $K(p)\cap G\neq \emptyset$ and since $p\in Y$, we infer from Lemma \ref{Lemma K(y)}
that $K(p)$ is a nonempty finite subset of $X$. Therefore, $\emptyset\neq K(p)\cap G\subseteq X\cap G$.
\end{proof}

\section*{Acknowledgements}
The author was partially supported by the NCN (National Science Centre, Poland) research Grant no. 2020/37/B/ST1/02613

\bibliographystyle{siam}
\bibliography{bib.bib}

\end{document}